\documentclass[12pt]{amsart}

\usepackage{amsmath, amssymb, amsthm}%
\usepackage{subfig}
\usepackage{tikz, tikz-3dplot}

\usepackage[backend=bibtex, backref=false, doi=false]{biblatex}
\usepackage[colorlinks, hyperfootnotes=false]{hyperref}

\usepackage{cleveref}

\theoremstyle{plain}
\newtheorem{lemma}{Lemma}[section]
\newtheorem{theorem}[lemma]{Theorem}
\newtheorem*{thm*}{Theorem}
\newtheorem{corollary}[lemma]{Corollary}
\newtheorem*{cor*}{Corollary}
\newtheorem{proposition}[lemma]{Proposition}

\theoremstyle{definition}
\newtheorem{definition}[lemma]{Definition}
\newtheorem{example}[lemma]{Example}
\newtheorem{remark}[lemma]{Remark}

\newcommand{\qq}[1]{``#1''}

\newcommand{\Th}[1]{$#1$\textsuperscript{th}}

\newcommand{\set}[1]{\left\{\,#1\,\right\}}
\newcommand{\with}{\ \vrule\ }
\newcommand{\defa}{\mathrel{\mathop{:}}=}

\newcommand{\kk}{\mathbb{K}}
\newcommand{\NN}{\mathbb{N}}
\newcommand{\RR}{\mathbb{R}}
\newcommand{\ZZ}{\mathbb{Z}}
\newcommand{\QQ}{\mathbb{Q}}

\newcommand{\gp}[1]{\ZZ #1}
\newcommand{\cone}[1]{\RR_+ #1}

\newcommand{\kQ}[1][Q]{\kk[#1]}
\newcommand{\kQsat}{\kQ[\sat{Q}]}
\newcommand{\kmod}[1]{\kk\{#1\}}
\newcommand{\kQsatm}{\kmod{\sat{Q}}}
\newcommand{\Quot}{\kQsatm/\kQ}
\newcommand{\Fsat}{F^{\sat{Q}}}
\newcommand{\pF}{\mathfrak{p}_F}
\newcommand{\sat}[1]{\overline{#1}}
\newcommand{\QoQ}{\sat{Q}\setminus Q}
\newcommand{\QoQF}{\sat{Q}_F \setminus Q_F}

\newcommand{\QuotF}{\kmod{\sat{Q}_F} / \kQ[Q_F]}

\newcommand{\Pc}{\mathcal{P}}
\newcommand{\Fc}{\mathcal{F}}

\DeclareMathOperator{\dep}{depth}
\newcommand{\sdep}{\,^*\!\!\dep}
\newcommand{\sdim}{\,^*\!\!\dim}

\DeclareMathOperator{\intt}{int}
\DeclareMathOperator{\Ass}{Ass}
\DeclareMathOperator{\Ann}{Ann}
\DeclareMathOperator{\Supp}{Supp}
\DeclareMathOperator{\supp}{\Supp}

\newcommand{\pp}{\mathfrak{p}}

\newcommand{\xx}{\mathbf{x}}
\newcommand{\mm}{\mathfrak{m}}
\newcommand{\lochom}[1]{H_\mm^{#1}(\kQ)}

\newcommand{\FF}[1]{\mathcal{F}_{#1}}

\newcommand{\nab}[1][q]{\nabla({#1})}
\newcommand{\nabv}[1][q]{\nab[#1]^\vee}
\newcommand{\snab}[1][q]{\sat{\nabla}({#1})}
\newcommand{\snabv}[1][q]{\snab[#1]^\vee}

\begin{document}
\title{Non-normal affine monoids}
\author{Lukas Katth\"an}
\address{Fachbereich Mathematik und Informatik, Philipps-Universit\"at Marburg}%
\email{katthaen@mathematik.uni-marburg.de}%
\thanks{This work was supported by the DAAD}

\date{\today}

\hypersetup{
pdftitle={Non-normal affine monoids},
pdfauthor={Lukas Katth\"an}
}

\begin{abstract}
	We give a geometric description of the set of holes in a non-normal affine monoid $Q$.
	The set of holes turns out to be related to the non-trivial graded components of the local cohomology of $\kQ$.
	From this, we see how various properties of $\kQ$ like local normality and Serre's conditions $(R_1)$ and $(S_2)$ are encoded in the geometry of the holes.
	A combinatorial upper bound for the depth the monoid algebra $\kQ$ is obtained and some cases where equality holds are identified.
	We apply this results to seminormal affine monoids.
\end{abstract}
\maketitle
%

\section{Introduction}
Let $Q$ be an affine monoid, i.e. a finitely generated submonoid of $\ZZ^N$ for some $N\in \NN$.
Further, let $\sat{Q}$ denote the normalization of $Q$.
In this paper, we give a geometric description of the set of holes $\QoQ$ in $Q$ and relate it to properties of $Q$.
Our first main result is the following.
\begin{thm*}[\Cref{thm:ZerlegungNeu}]

	Let $Q$ be an affine monoid. There exists a (not necessarily disjoint) decomposition
	\begin{equation}\label{eq:zerlegung}
	\QoQ = \bigcup_{i=1}^l (q_i + \gp{F_i}) \cap \cone{Q}
	\end{equation}
	with $q_i \in \sat{Q}$ and faces $F_i$ of $Q$. If the union is irredundant (i.e. no $q_i + \gp{F_i}$ can be omitted), then the decomposition is unique. 
\end{thm*}
We call a set $q_i + \gp{F_i}$ from \eqref{eq:zerlegung} a \emph{$j$-*dimensional family of holes}, where $j$ is the *dimension of $F$.
The faces appearing in \eqref{eq:zerlegung} correspond to the associated primes of the quotient $\Quot$. 
The same face may appear several times in \eqref{eq:zerlegung}, in fact, the number of times a face appears equals the multiplicity of the corresponding prime.

The significance of this result is that the decomposition can easily be identified from pictures. Because the decomposition is unique, there is no way to pick a wrong decomposition. 
This can be used to easily read off several properties of $Q$:
\begin{thm*}[\Cref{thm:compos}]
	Let $Q$ be an affine monoid of *dimension $d$. The following holds:
	\begin{itemize}
		\item If $d \geq 2$, then $\sdep \kQ = 1$ if and only if there is a $0$-*dimensional family of holes.
		\item $Q$ is locally normal                      if and only if there is no family of holes of positive *dimension.
		\item $\kQ$ satisfies Serre's condition $(R_1)$ if and only if there is no family of holes of *dimension $d-1$.
		\item $\kQ$ satisfies Serre's condition $(S_2)$ if and only if every family of holes has *dimension $d-1$.
	\end{itemize}
\end{thm*}

We also establish a close relation of the local cohomology of $\kQ$ at the *maximal ideal to the families of holes that is summarized in the next result.
\begin{thm*}[\Cref{thm:lochomholes}] 
	Let $q \in \gp{Q}$ such that $q \notin - \intt \sat{Q}$.
	If $\lochom{i+1}_q \neq 0$ for some $i$,
	then $q$ is contained in a family of holes of *dimension at least $i$.
	On the other hand, every $i$-*dimensional family of holes contains an element $q \in \gp{Q}$
	such that $\lochom{i+1}_q \neq 0$.
\end{thm*}
The hypothesis that $q \notin - \intt \sat{Q}$ is not an essential restriction, as we will see in the discussion leading to the result above.
This result leads to an upper bound on the *depth of $\kQ$, that is easily computable in many cases.
\begin{thm*}[\Cref{thm:main1}]
	If $Q$ has an $i$-*dimensional family of holes, then the *depth of $\kQ$ is at most $i+1$.
\end{thm*}
This theorem states that a non-normal affine monoid with \qq{few} holes has a low *depth. This is somewhat counterintuitive, because Hochster's Theorem (\cite[Theorem 6.10]{brunsgubel})  states that the absence of holes, i.e. normality, implies maximal *depth.
In small examples, it is often not too difficult to determine the bound given by this theorem geometrically.
This can be easier than to compute the actual *depth algebraically.
In general, the *depth may be strictly smaller than the bound given by \Cref{thm:main1}.
However, in \Cref{thm:main2}, \Cref{thm:disjoint} and \Cref{thm:S2} we identify some special cases where equality holds.
See \Cref{ex:graph} for an application.

\begin{figure}[t]
	\captionsetup[subfloat]{width=4.2cm, singlelinecheck=true}
	\subfloat[The decomposition of \cite{brunsgubel}]{
		\begin{tikzpicture}[scale=0.7, radius=0.1667]
		\draw[<->] (7.5,0) -- (0,0) -- (0,7.5);
		
		\foreach \x in {0,2,3,...,6}
		\foreach \y in {0,2,3,...,6}
			\fill (\x,\y) circle;
		
		\foreach \x in {0,2,3,...,6}
			\draw (\x,1) circle;
		
		\foreach \y in {0,1,...,6}
			\draw (1,\y) circle;

		\foreach \pos in {(0,1),(1,0),(1,1),(1,3),(3,1)}
			\draw[densely dotted] \pos circle[radius=0.3333];
		
		\draw[densely dotted]
			(0.5 ,7   ) --
			(0.5 ,4   ) .. controls (0.5 ,3.8 ) and (0.7 ,3.5 ) ..
			(1   ,3.5 ) .. controls (1.3 ,3.5 ) and (1.4 ,3.2 ) ..
			(1.4 ,3   ) .. controls (1.4 ,2.8 ) and (1.3 ,2.33) ..
			(1   ,2.33) .. controls (0.8 ,2.33) and (0.66,2.2 ) ..
			(0.66,2   ) .. controls (0.66,1.8 ) and (0.8 ,1.66) ..
			(1   ,1.66) .. controls (1.3 ,1.66) and (1.5 ,1.8 ) ..
			(1.5 ,2   ) --
			(1.5 ,7   );
		
		\begin{scope}[cm={0,1,1,0,(0,0)}]
		\draw[densely dotted]
			(0.5 ,7   ) --
			(0.5 ,4   ) .. controls (0.5 ,3.8 ) and (0.7 ,3.5 ) ..
			(1   ,3.5 ) .. controls (1.3 ,3.5 ) and (1.4 ,3.2 ) ..
			(1.4 ,3   ) .. controls (1.4 ,2.8 ) and (1.3 ,2.33) ..
			(1   ,2.33) .. controls (0.8 ,2.33) and (0.66,2.2 ) ..
			(0.66,2   ) .. controls (0.66,1.8 ) and (0.8 ,1.66) ..
			(1   ,1.66) .. controls (1.3 ,1.66) and (1.5 ,1.8 ) ..
			(1.5 ,2   ) --
			(1.5 ,7   );
		\end{scope}
		\end{tikzpicture}
	} %
	\qquad
	\subfloat[Our decomposition]{
		\begin{tikzpicture}[scale=0.7, radius=0.1667]
		\draw[<->] (7.5,0) -- (0,0) -- (0,7.5);
		
		\foreach \x in {0,2,3,...,6}
		\foreach \y in {0,2,3,...,6}
			\fill (\x,\y) circle;
		
		\foreach \x in {0,2,3,...,6}
			\draw (\x,1) circle;
		
		\foreach \y in {0,1,...,6}
			\draw (1,\y) circle;
		
		\draw[densely dotted]
			(0.66,7) -- +(0,-7)

			arc[start angle=180, end angle=360, x radius=0.33, y radius=0.33] -- +(0,7)
			(7, 0.66) -- +(-7,0)
			arc[start angle=270, end angle=90, x radius=0.33, y radius=0.33] -- +(7,0);
		\end{tikzpicture}
	}
	\caption{Different decomposition of the holes of a $2$-dimensional affine monoid}
	\label{fig:2dim}
\end{figure}
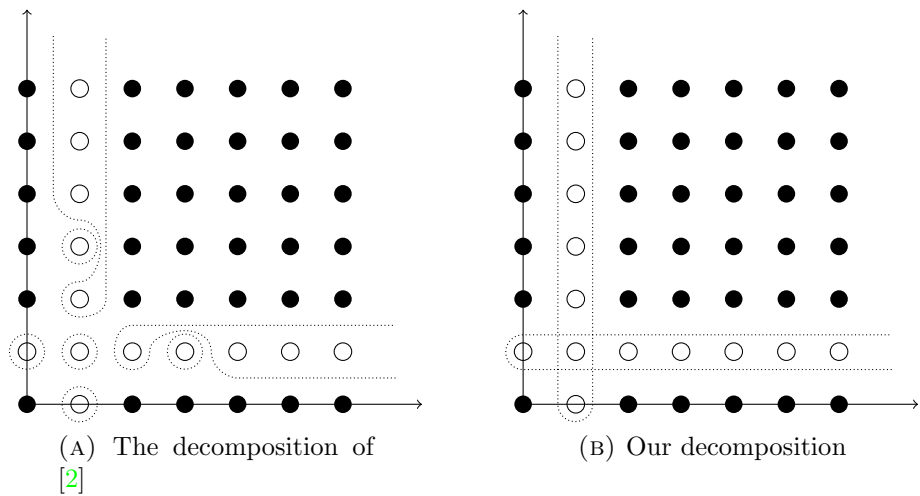
In \cite[Prop. 2.35]{brunsgubel} a different decomposition of the holes is considered.
It is shown in \cite{brunsgubel} that one can always find a decomposition of $\QoQ$ into a disjoint union of translates of faces of $Q$:
\begin{equation}\label{eq:bgzerlegung}
	\QoQ = \bigcup_{i=1}^l q_i + F_i
\end{equation}
In fact, this statement and its proof have been the motivation for proving \Cref{thm:ZerlegungNeu}.
\Cref{fig:2dim} shows an example of both kinds of decompositions.
The decomposition given in \eqref{eq:bgzerlegung} is disjoint, but far from being unique.
Moreover, it seems difficult to extract algebraic information about $Q$ from this decomposition.
 
In the last part of this paper, we apply our results to seminormal affine monoids.
For this class monoids, we give a new proof of the cohomological characterization of seminormality of \cite{blrSeminorm}.
While our proof is not actually simpler than the original one, we believe that it offers a new, more geometric perspective.
Moreover, we extend this and some other results of \cite{blrSeminorm} to the non-positive case.

\section{Preliminaries and notation}\label{sec:prelim}
An affine monoid $Q$ is a finitely generated submonoid of the additive monoid $\ZZ^N$ for an $N \in \NN$. 
For general information about affine monoids see \cite{brunsgubel} or \cite{millersturm}.
We denote the group generated by $Q$ by $\gp{Q}$, the convex cone generated by $Q$ by $\cone{Q} \subseteq \RR^N$ and the normalization by $\sat{Q} = \gp{Q} \cap \cone{Q}$. 
Recall that an element $q \in \gp{Q}$ is contained in $\sat{Q}$  if and only if a multiple of $q$ lies in $Q$.
%
A \emph{face} $F \subseteq Q$ of $Q$ is a subset such that for $a,b \in Q$ the following holds:
\[ a+b \in F \qquad \Longleftrightarrow \qquad a,b \in F \]
$Q$ has a unique minimal face $F_0$ that contains exactly the invertible elements.
We call $Q$ \emph{positive} if the only invertible element is $0$.
For every element $q \in Q$, there exists a unique minimal face $F$ containing $q$.
We say that $q$ is an \emph{interior point} of $F$ and write $\intt F$ for the set of interior points of $F$.
Note that by definition $0 \in \intt F_0$.
%
The dimension of a face $F$ is the rank of $\gp{F}$. Since we are working with not-necessarily positive affine monoids, it is more convenient to consider a normalized version of the dimension.
So we define the \emph{*dimension} as $\sdim Q \defa \dim Q - \dim F_0$, and $\sdim F \defa \dim F - \dim F_0$ for every face $F$ of $Q$.
For a field $\kk$, we write $\kQ$ for the monoid algebra of $Q$.
Further, for an element $q \in Q$, we write $\xx^q \in \kQ$ for the corresponding monomial.
For a face $F$ we define $\pF \subseteq \kQ$ to be the vector space generated by those monomials $\xx^q$ such that $q \in Q\setminus F$.
Then $\pF$ is a monomial prime ideal of $\kQ$ and 
 $\kQ/\pF \cong \kQ[F]$.
Thus the ideal $\pp_{F_0}$ associated to the minimal face is the unique maximal graded ideal of $\kQ$ with respect to the natural $\gp{Q}$-grading.
We will sometimes write $\mm$ for this ideal.
Its height equals the maximal length of a descending chain of faces of $Q$, so $(\kQ, \mm)$ is a *local ring of *dimension $\sdim Q$.
More general, the height of $\pF$ equals $\sdim F$ for every face $F$.
%
The \emph{*depth} of $\kQ$ is the maximal length of a regular sequence in $\mm$. Equivalently, it is the depth of the (inhomogeneous) localization $\kQ_{\mm}$.
For a face $F$ of $Q$, we denote by
\[ Q_F \defa \set{q-f \with q\in Q, f \in F} \]
the \emph{localization} of $Q$ at $F$. 
It holds that $\kQ[Q_F] = \kQ_{(\pF)}$, where the later is the homogeneous localization of $\kQ$ at $\pF$.
Note that localizations are almost never positive.
The faces of $Q$ are in bijection with the faces of $\sat{Q}$, but as sets they may be different.
Therefore, for a face $F$ of $Q$, we write $\Fsat \defa \set{q \in \sat{Q} \with \exists n\in \NN: nq \in F}$ for the corresponding face of $\sat{Q}$. 
\begin{lemma}\label{prop:norloccom}
	Normalization and localization commute.
	More precisely, if $F \subset Q$ is a face, then it holds that $\sat{(Q_F)} = (\sat{Q})_F$.
	Moreover, it makes no difference if we localize $\sat{Q}$ as a $Q$-module or as an affine monoid on its own:
	$(\sat{Q})_F = (\sat{Q})_{\Fsat}$.
\end{lemma}
\begin{proof}
	The equality $\sat{(Q_F)} = (\sat{Q})_F $ follows from the corresponding algebraic statement, see \cite[Prop. 4.13]{eisenbud}.
	Further, $(\sat{Q})_F = (\sat{Q})_{\Fsat}$, because $F$ contains an interior point of $\Fsat$.
\end{proof}

For a $\gp{Q}$-graded $\kQ$-module $N$, the \emph{support} of $N$, $\supp N$, is defined to be the set of those $q \in \gp{Q}$, such that there exists an element of degree $q$ in $N$.
\begin{definition}
Let $Q$ be an affine monoid. We call $Q$ \emph{locally normal} if every localization $Q_F$ at a face $F \neq F_0$ is normal.
\end{definition}
Since localizations of normal affine monoids are again normal, it is enough to consider faces of *dimension $1$. 
Note that a polytopal affine monoid is locally normal if and only if it is very ample.

Finally, a \emph{cross section polytope} of a polyhedral cone $C \subset \RR^d$ of a polytope whose face lattice is isomorphic to the face lattice of $C$.
Such a polytope always exists and can be obtained by intersection $C$ with a suitable affine subspace. 
Note that if $Q$ is an affine monoid, then its face lattice is isomorphic to the face lattice of a cross section polytope of $\cone{Q}$.

\section{The structure of the set of holes}
In this section, we describe the structure of the set of holes $\QoQ$.
Following an idea from \cite[p. 139]{brunsgubel}, we consider more generally graded subquotients of $\kQ[\gp{Q}]$.

\begin{theorem}\label{thm:ZerlegungNeu}
Let $U \subseteq M \subseteq \kQ[\gp{Q}]$ be finitely generated graded $\kQ$-modules and let $N \defa M / U$. For each $\pp \in \Ass N$, let $r_\pp$ denote its multiplicity on $N$ and let $F_\pp$ denote the corresponding face of $Q$. Then
\begin{equation}\label{eq:zerlegungVoll}
	\supp N = \bigcup_{\pp \in \Ass N} \bigcup_{i=1}^{r_\pp} \gp{F_\pp} + q_{\pp,i} \cap \supp M
\end{equation}
for elements $q_{\pp,i} \in \gp{Q}$. This union is irredundant and it is unique (as a union of sets).
\end{theorem}
The most important special case for this paper is of course $M = \kQsat$ and $U = \kQ$, where $\Supp N = \QoQ$.
In this case, we will call a set $\gp{F} + q$ in \eqref{eq:zerlegungVoll} a \emph{family of holes}.
Another noteworthy special case is $Q = \NN^n, M = \kQ = \kk[x_1, \dotsc, x_n]$ and $U$ a squarefree monomial ideal.
If $\Delta$ denotes the simplicial complex corresponding to $U$, the faces $F$ in \eqref{eq:zerlegungVoll} correspond to the facets of $\Delta$ and thus \eqref{eq:zerlegungVoll} corresponds to the primary decomposition of its Stanley-Reisner ideal $U$.

The decomposition of \Cref{thm:ZerlegungNeu} behaves nicely under localization. For simplicity, we state this only in a special case:
\begin{corollary}\label{prop:locholesNeu}
	Let $F$ be a face of $Q$.
	The families of holes of $Q_F$ are exactly those families of holes $\gp{G} + q $ of $Q$ which satisfy $F \subset G$.
	In particular, $Q_F$ is normal if and only if there is no family of holes $\gp{G}+q$ with $F \subseteq G$. 
\end{corollary}
We will prove the corollary below.
Before giving the proof of \Cref{thm:ZerlegungNeu}, we prepare two lemmata.
First, we have a variant of the well-known fact that a vector space over an infinite field cannot be written as a union of finitely many subspaces.
\begin{lemma}\label{lemma:vrunion}
	Let $V$ be a vector space over $\QQ$ and $C \subseteq V$ be a convex cone. 
	If $C$ contains a generating set of $V$, then it is not contained in any finite union of proper subspaces of $V$. 
\end{lemma}
\noindent The proof is similar to the statement about vector spaces. Next we need a discrete analogue of the preceding lemma.
\begin{lemma}\label{lemma:union}
	Let $q, p_1, \dotsc, p_l \in \gp{Q}$ be lattice points and let $F, G_1, \dotsc, G_l$ be (not necessarily distinct) faces of $Q$. If $F + q$ is contained in the union $\bigcup_i \gp{G_i} + p_i$, then it is already contained in one of the sets $\gp{G_i} + p_i$ and it holds that $F \subseteq G_i$ for that index $i$.
\end{lemma}
Note that this Lemma does not hold for arbitrary subgroups of $\ZZ^N$, for example $\ZZ = 2\ZZ \cup (2\ZZ + 1)$.
\begin{proof}
	We may assume that $F + q$ has a non-empty intersection with every $\gp{G_i} + p_i$ for $1\leq i \leq l$. 
	If $F \subseteq G_i$ for any $i$, then $F + q \subset \gp{F} + q' \subset \gp{G_i} + q' = \gp{G_i} + p_i$ for $q' \in F + q \cap \gp{G_i} + p_i$. Thus in this case our claim holds. We will show that there exists always an $i$ such that $F \subseteq G_i$.
	
	Assume that $F \nsubseteq G_i$ for every $i$. As a notation, for a subset $S \subset \QQ^N$, we write $\QQ S$ for the $\QQ$-subspace generated by $S$.
	Then, $\QQ(\gp{F} \cap \gp{G_i}) \subsetneq \QQ F$ for every $i$. Indeed, it holds that $\QQ(\gp{F} \cap \gp{G_i}) \subseteq \QQ F \cap \QQ G_i \subseteq \QQ F$. The second inclusion is strict except in the case that $\QQ F \subseteq \QQ G_i$. But this would imply that $F \subseteq G_i$, because $F = \QQ F \cap Q$ and $G_i = \QQ G_i \cap Q$. Here we use that $F$ and $G_i$ are faces of a common affine monoid.
	
	By \Cref{lemma:vrunion}, we can find an element $\hat{p}$ in the cone generated by $F$ that is not contained in any $\QQ(\gp{F} \cap \gp{G_i})$. By multiplication with a positive scalar, we can assume $\hat{p} \in F$. For every non-negative integer $\lambda$, it holds that $\lambda \hat{p} + q \in F + q \subset \bigcup_i \gp{G_i} + p_i$. Since the union is finite, there exists an index $i$ and two different integers $\lambda, \lambda' \in \ZZ$ such that $\lambda \hat{p} + q,\lambda' \hat{p} + q \in \gp{G_i} + p_i$.
	But now it follows that $(\lambda - \lambda')\hat{p} \in \gp{F} \cap \gp{G_i}$ and thus $\hat{p} \in \QQ(\gp{F} \cap \gp{G_i})$, a contradiction to our choice of $\hat{p}$.
\end{proof}

\begin{proof}[Proof of \Cref{thm:ZerlegungNeu}]
\textbf{The 1\textsuperscript{st} step.}
	We start by proving the following fact: If $\pp \subset \kQ$ is a homogeneous prime ideal and $q \in \supp N_{(\pp)} \cap \supp M$, then the natural map $(M)_q \rightarrow (N)_q \rightarrow (N_{(\pp)})_q$ is an isomorphism of ($1$-dimensional) $\kk$-vector spaces.
	Here, $(M)_q$ denotes the graded component of $M$ in degree $q$.

	To see this, let $0 \neq \frac{n}{s} \in N_{(\pp)}$ be a homogeneous element of degree $q$. 
	Note that all graded components of $M$ are $1$-dimensional, therefore $n \neq 0$ implies that $\deg n \notin \supp U$.
	There exists an element $0 \neq m \in M$ of degree $q$ such that $s m + U = n$. Since $\deg (s m) = \deg n \notin \supp U$, it follows that $\deg m = q \notin \supp U$.
	So $n' \defa m + U$ does not go to zero in the quotient $N = M/U$. Moreover, $s \frac{n'}{1} = \frac{s n'}{1} = \frac{n}{1} \neq 0$ in $N_\pp$, hence $\frac{n'}{1} \neq 0$.

\textbf{The 2\textsuperscript{nd} step.}
	First, we proof the following preliminary decomposition:
	\begin{equation}\label{eq:zerlegungH0}
		\supp N = \bigcup_{\pp \in \Ass N} \supp H^0_\pp(N_{(\pp)}) \cap \supp M
	\end{equation}
	\begin{itemize}
	\item[\qq{$\subseteq$}:] 
		Let $m \in N$ be a homogeneous element and let $\pp$ be a minimal prime over the annihilator of $m$.
		Then $\pp$ is an associated prime of $N$, because it is a minimal prime of the submodule generated by $m$.
		Moreover, $m$ does not go to zero in the localization $N_{(\pp)}$.
		Since $m \in H^0_\pp(N_{(\pp)})$, the degree of $m$ is contained in the right hand side of \eqref{eq:zerlegungH0}.

\item[\qq{$\supseteq$}:]
	For $\pp \in \Ass N$ and $q \in \supp H^0_\pp(N_{(\pp)}) \cap \supp M$ it follows immediately from the first step that $(N)_q \neq 0$.
\end{itemize}

\textbf{The 3\textsuperscript{rd} step.}
	Now fix a $\pp \in \Ass N$ and let $F \defa F_\pp$ and $r \defa r_\pp$. We show that there are $q_1, \dotsc, q_r \in \supp N$ such that
	\begin{equation*}
		\supp H^0_\pp(N_{(\pp)}) = \bigcup_{i=1}^r \gp{F} + q_i.
	\end{equation*} 
	The module $V \defa H^0_\pp(N_{(\pp)})$ is of finite *length $r$ over the *local ring $\kQ_{(\pp)} = \kQ[Q_F]$, so there exists a *composition series
	$ 0 = V_0 \subset V_1 \subset \dots \subset V_r = V $ 
	such that $V_i / V_{i-1} \cong \kQ[Q_F] / \pF (-q_i) = \kQ[\gp{F}](-q_i)$ for $1 \leq i \leq r$ and certain $q_i \in \supp V$.
	Here, $(.)(-q_i)$ denotes a shift in the grading by $q_i$. Hence
	\[ \supp V = \bigcup_{i=1}^r \supp \kQ[\gp{F}](-q_i) = \bigcup_{i=1}^r \gp{F} + q_i. \]

\textbf{The 4\textsuperscript{th} step.}
	Finally, we show that the union in \eqref{eq:zerlegungVoll} is unique and thus irredundant.
	For this, we consider any decomposition of $\supp N$ of the form
	\begin{equation}\label{eq:second}
		\supp N = \bigcup_j \gp{G_i} + p_i \cap \supp M
	\end{equation}
	and show that each set $\gp{F} + q$ of \eqref{eq:zerlegungVoll} equals one of the sets in \eqref{eq:second}.
	
	So fix a set $\gp{F} + q$ occurring in \eqref{eq:zerlegungVoll}.
	By applying \Cref{lemma:union} twice, we see that there exists an index $i_0$ and a second set $\gp{F'} + q'$ in \eqref{eq:zerlegungVoll} such that $\gp{F} + q  \subseteq \gp{G_{i_0}} + p_{i_0} \subseteq \gp{F'} + q'$ and $F \subseteq G_{i_0} \subseteq F'$. 

	Let $0 \neq n \in N$ be an element such that $\deg n \in \gp{F} + q$.
	By the first step, the image of $n$ is nonzero in both $N_{(\pp_F)}$ and $N_{(\pp_{F'})}$.
	Recall that $\gp{F} + q \subset\supp H^0_{\pp_F}(N_{(\pp_F)})$ by our construction of \eqref{eq:zerlegungVoll} via \eqref{eq:zerlegungH0}.
	Because each homogeneous component of $N_{(\pF)}$ is one-dimensional, it follows that $n \in H^0_{\pp_F}(N_{(\pp_F)})$ and $\pp_F$ is thus a minimal prime over the annihilator of $n$ in $N$ (cf. \cite[2.19]{eisenbud}). But by the same argument $\pp_{F'}$ is also a minimal prime over the annihilator. Since $\pp_{F'} \subseteq \pp_F$, it follows that $\pp_{F'} = \pp_F$.
	But this already implies that $\gp{F} +q = \gp{G_{i_0}} + p_{i_0} = \gp{F'} + q'$, so the claim is proven.
\end{proof}

Note that in the second step above proof, we do not use the assumption $M \subseteq \kQ[\gp{Q}]$ for the  inclusion \qq{$\subseteq$}. However, for the other inclusion this assumption is crucial. For example, consider $Q = \NN_0$, $M = \kQ e_1 \oplus \kQ(-1) e_2$ and $U = \kQ e_1$, where $e_1$ and $e_2$ are free generators. Then $\supp(M/U) = \set{1,2,\dotsc}$, while $\supp H^0_{(0)}((M/U)_{((0))}) \cap \supp M =  \set{0,1,2,\dotsc}$.

\begin{proof}[Proof of \Cref{prop:locholesNeu}]
Let $N \defa \Quot$.
The associated primes of $N_{(\pF)}$ are exactly those associated primes of $N$ that are contained in $\pF$. Moreover, for such a prime $\pp \in \Ass N_{(\pF)}$, it holds that $H_\pp^0(N_{(\pp)}) = H_\pp^0((N_{(\pF)})_{(\pp)})$. So the claim is immediate from the proof of \ref{thm:ZerlegungNeu}.
\end{proof}

\section{Local cohomology and holes}
In this section, we consider the local cohomology of the monoid algebra $\kQ$ with support on the maximal graded ideal $\mm \defa \pp_{F_0}$. Recall that the local cohomology can be computed by the Ishida complex \cite{ishida} as follows:
Consider the $\gp{Q}$-graded complex
\[ \mho_Q:\, 0 \rightarrow \kQ \rightarrow \bigoplus_{F \in \FF{1}} \kQ[Q_F] \rightarrow \cdots 
\rightarrow \bigoplus_{F \in \FF{d-1}} \kQ[Q_F] \rightarrow \kQ[\gp{Q}] \rightarrow 0 \]
where $\FF{i}$ denotes the set of $i$-*dimensional faces of $Q$.
The maps are given by $\delta_i: \kQ[Q_F] \ni \xx^q \mapsto \sum_{G \supset F} \epsilon(F,G) \xx^q$ via the canonical inclusion $\kQ[Q_F] \rightarrow \kQ[Q_G]$ for $F \subset G$, and $\epsilon(F,G)$ is an appropriate sign function.
See \cite[Section 13.3]{millersturm} for the exact definition.
The cohomological degrees are chosen such that the modules $\kQ$ and $\kQ[\gp{Q}]$ sit in degree $0$ respectively $d$.

\begin{theorem}[Thm. 13.24, \cite{millersturm}]
	The local cohomology of any $\kQ$-module $M$ supported on $\mm$ is the cohomology of the Ishida complex tensored with $M$:
	\[ H_\mm^i(M) \cong H^i(M \otimes \mho_Q) \]
	The isomorphism respects the $\gp{Q}$-grading.
\end{theorem}
\noindent We use the Ishida complex to relate the local cohomology of $\kQ$ to the local cohomology of $\Quot$.
\begin{theorem}\label{thm:lochomneu}
	Let $Q$ be an affine monoid of *dimension $d$, $i \leq d$ and integer and let $q \in \gp{Q}$.
	Then the following holds:
	\begin{enumerate}
		\item If $i < d$, then $\lochom{i} \cong H_{\mm}^{i-1}(\Quot) \,.$
		\item If $i=d$ and $q \notin -\intt \sat{Q}$, then $\lochom{i}_q \cong H_{\mm}^{i-1}(\Quot)_q$ as $\kk$-vector spaces.
		\item If $q \in -\intt \sat{Q}$, then
		\[
		\lochom{i}_q = %
		\begin{cases}
		\kk &\text{ if } i = d\,, \\
		0 &\text{ otherwise.}
		\end{cases}
		\]
	\end{enumerate}
\end{theorem}
\begin{proof}
	First, we compute the local cohomology of $\kQsatm$. In this proof, we write $\mm = \mm_{\kQ}$ for the *maximal ideal of $\kQ$ and $\mm_{\kQsat}$ for the *maximal ideal of $\kQsat$.
	It is well-known (cf. \cite[Thm 6.3.4]{brunsherzog}) that 
	\[ H^i_{\mm_{\kQsat}}(\kQsat)_q =
		\begin{cases}
			\kk &\text{ if } i=d \text{ and } q \in - \intt \sat{Q}\,, \\
			0 &\text{ otherwise.}
		\end{cases}
	\]
	Note that the extension of the *maximal ideal of $\kQ$ to $\kQsat$ is $\mm_{\kQsat}$-primary. Therefore, by the graded Independence Theorem \cite[Thm 14.1.7]{brodmannsharp2} it holds that
	\[H^i_{\mm_{\kQ}}(\kQsatm)_q \cong H^i_{\mm_{\kQ}\kQsat}(\kQsat)_q = H^i_{\mm_{\kQsat}}(\kQsat)_q\]
	as vector spaces for $i\in\NN$ and $q\in\gp{Q}$.
	Next, we consider the short exact sequence
	\[ 0 \rightarrow \kQ \rightarrow \kQsatm \rightarrow \Quot \rightarrow 0 \]
	The corresponding long exact sequence in cohomology gives immediately that $\lochom{i} \cong  H_{\mm}^{i-1}(\Quot)$ for $i < d$.
	
	For $i = d$, we make a case distinction.
	If $q \notin - \intt \sat{Q}$, then $H_{\mm}^{i}(\kQsatm)_q = 0$ by the discussion above.
	Hence one can read off from the long exact sequence that $\lochom{d}_q \cong H_{\mm}^{d-1}(\Quot)_q$, because the maps are homogeneous.
	On the other hand, if $q \in - \intt \sat{Q}$, then $q \notin Q_F$ for any face $F$ of $Q$.
	So the degree $q$ part of $\mho_Q$ is just $0 \rightarrow \kk \rightarrow 0$ with the $\kk$ in cohomological degree $d$.
\end{proof}
\begin{corollary}\label{prop:dep}
	If $Q$ is not normal, then $\sdep \kQ = \sdep \Quot + 1$.
\end{corollary}
\begin{proof}
	Recall that the *depth of a *local ring $(R, \mm)$ equals the minimal degree $i$ such that $H_\mm^i(R) \neq 0$%
	\footnote{This is well-known for local rings (cf. \cite[Theorem 3.5.7]{brunsherzog}) and the *local case can be reduced to the local case by localizing at the *maximal ideal.}%
	.
	Hence the claim is immediate from \Cref{thm:lochomneu}.
\end{proof}

Our next goal is to show how the families of holes restrict the support of the local homology modules.
For this, we take a closer look at the Ishida complex.
In this we follow \cite[Section 12.2]{millersturm}.
Fix an element $q \in\gp{Q}$.
The part of $\mho_Q$ in degree $q$ is determined by the faces $F \subset Q$ such that $q \in Q_F$.
Therefore, we consider the set $\nab \defa \set{F \subseteq Q \with q \in Q_F, F \text{ a face}}$.
This set is clearly closed under going up in the face lattice of $Q$.
Now let $\Pc$ be a cross-section polytope of $\cone{Q}$ and let $\Pc^\vee$ be the polar polytope of $\Pc$.
Then the face lattice of $\Pc^\vee$ equals the order dual of the face lattice of $Q$ (i.e. the face lattice of $\Pc$ turned upside down).
Hence the images $\nabv$ of the faces in $\nab$ in the face lattice of $\Pc^\vee$ form a set that is closed under going down.
In other words, $\nabv$ is a polyhedral subcomplex of the boundary complex of $\Pc^\vee$.
Because $\nab$ corresponds to the part of $\mho_Q$ in degree $q$, we can reinterpret this part as an (augmented) polyhedral chain complex for $\nabv$, while reversing the cohomological degrees.
So the reduced homology of the polyhedral cell complex $\nabv$ gives us the local cohomology of $\kQ$ in degree $q$ (\cite[p. 258]{millersturm}):
\begin{equation}\label{eq:nabla}
\lochom{i}_q = \tilde{H}_{d-1-i}(\nabv, \kk)
\end{equation}
\begin{theorem}\label{thm:lochomholes} 
	Let $q \in \gp{Q}$ such that $q \notin - \intt \sat{Q}$.
	If $\lochom{i+1}_q \neq 0$ for some $i$,
	then $q$ is contained in a family of holes of *dimension at least $i$.
	On the other hand, every $i$-*dimensional family of holes contains an element $q \in \gp{Q}$
	such that $\lochom{i+1}_q \neq 0$.
\end{theorem}
\begin{proof}
	First, assume that $\lochom{i+1}_q \neq 0$ for some $i$.
	By \Cref{thm:lochomneu}, we have $\lochom{i+1}_q \cong H_{\mm}^{i}(\Quot)_q$.
	We consider the \Th{i} module in $\mho_Q \otimes \Quot$. It is
	\[ \bigoplus_{F \in \FF{i}} \kQ[Q_F] \otimes \Quot = \bigoplus_{F \in \FF{i}} \kmod{\sat{Q}_F} / \kQ[Q_F] \]
	If $\lochom{i+1}_q \neq 0$, then there is an element of degree $q$ in this module.
	Hence there is a face $F$ of *dimension $i$ such that $q \in \QoQF$.
	Now our description of the holes in the localization $Q_F$ (cf. \Cref{prop:locholesNeu}) implies that $q$ is contained in a family of holes $\gp{G} + p$ of $Q$ with $F \subset G$. In particular, $\sdim G \geq \sdim F = i$.

	Next, let $\gp{F} + p$ be an $i$-*dimensional family of holes.
	Let $q \defa p - m q'$ for $q' \in \intt F$ arbitrary and an $m \in \NN$. 
	If $m$ is large enough, then $\sigma_G(q) < 0$ for each facet $G$ not containing $F$.
	We claim that $\nab = \set{G \with G \supsetneq F}$.
	Thus $\nabv$ is the boundary complex of the face $\bar{F}$ corresponding to $F$ in the polar polytope $\Pc^\vee$.
	This is a sphere of dimension $\dim \bar{F} - 1 = \dim \Pc -1 - (\sdim F - 1) - 1 = d - 2 - i$.
	So by \eqref{eq:nabla} it follows
	\[
	\lochom{i+1}_q = \tilde{H}_{d-2-i}(S^{d-2-i}, \kk) = \kk \]
	To prove our claim, we first consider a face $G$ that does not contain $F$.
	For such a $G$ we can find a facet $G' \supset G$ that does not contain $F$.
	By our construction, $\sigma_{G'}(q)$ and hence $q \notin Q_{G'}$.
	Thus $q \notin Q_G$ and therefore $G \notin \nab$.
	Next, by our choice of $q$, it holds that $q \in \QoQF$.
	In particular $F \notin \nab$.
	Moreover, $q \in \sat{Q}_G$ for every $G \supset F$, because $\sat{Q}_G \supset \sat{Q}_F$.
	It remains to show that $q \in Q_G$ for every $G \supsetneq F$.
	So assume on the contrary that $q \in \sat{Q}_G \setminus Q_G$ for such a $G$.
	There exists an element $f \in \intt F$ such that $q + f \in \sat{Q}_G \setminus Q_G \cap \sat{Q}$.
	But this implies $G + q + f \subset \QoQ$, which contradicts our choice $q \in \gp{F} + p$, by \Cref{thm:ZerlegungNeu}.
\end{proof}

We give an example to demonstrate the geometric meaning of the results in this section.
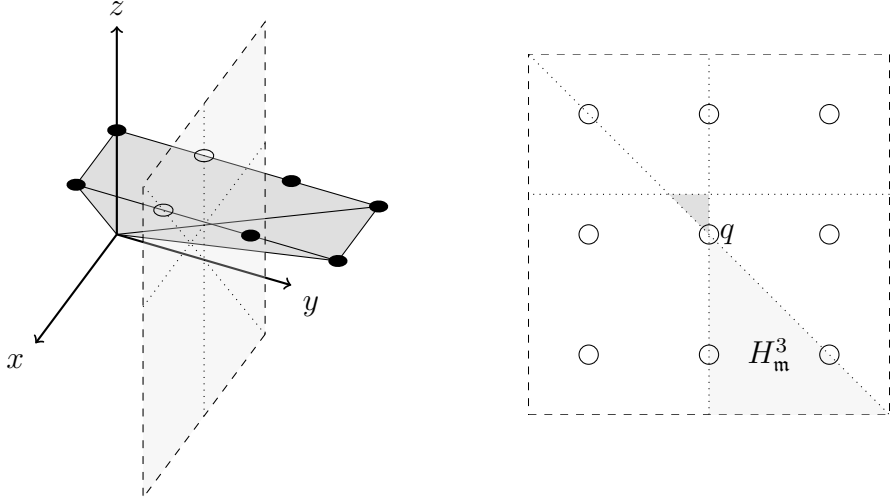
\begin{figure}[t]
	\centering
	\tdplotsetmaincoords{60}{115}
	\begin{tikzpicture}[scale=1.6]
	
	\begin{scope}[tdplot_main_coords, xscale=.8]
	\newcommand{\fac}{1}
	\coordinate (a) at ($(0,0,0)!\fac!(0,0,1)$);
	\coordinate (b) at ($(0,0,0)!\fac!(1,0,1)$);
	\coordinate (c) at ($(0,0,0)!\fac!(1,3,1)$);
	\coordinate (d) at ($(0,0,0)!\fac!(0,3,1)$);
	
	\fill[gray!50, opacity=0.5] (0,0,0) -- (a) -- (d) -- cycle;
	\fill[gray!50, opacity=0.5] (0,0,0) -- (a) -- (b) -- cycle;
	\fill[gray!50, opacity=0.5] (0,0,0) -- (c) -- (d) -- cycle;
	
	\draw (a) --(b) -- (c) -- (d) -- cycle;
	
	\foreach \x in {0,1}
	\foreach \y in {0,3}
	\draw (0,0,0) -- ($(0,0,0)!\fac!(\x,\y,1)$);
	
	\draw[thick,->] (0,0,0) -- (2,0,0) node[anchor=north east]{$x$};
	\draw[thick,->] (0,0,0) -- (0,2,0) node[anchor=north west]{$y$};
	\draw[thick,->] (0,0,0) -- (0,0,2) node[anchor=south]{$z$};
	
	\draw[dashed, fill=gray!20, fill opacity=0.3] (-1.5,1,-1.5) -- (1.5,1,-1.5) -- 
	(1.5,1,1.5) -- (-1.5, 1,1.5) -- cycle;
	\draw[dotted] (-1.5,1,-1.5) -- (1.5,1,1.5)
	(0,1,-1.5) -- (0, 1,1.5)
	(-1.5,1,0.3333) -- (1.5,1,0.3333);
	
	\foreach \x in {0,1}
	\foreach \y in {0,2,3}
	\fill (\x, \y, 1) circle (0.1);
	
	\foreach \x in {0,1}
	\draw (\x, 1, 1) circle (0.1);
	
	\end{scope}
	\begin{scope}[xshift=140, scale=1, xscale=-1]
	\clip (-1.5, -1.5) rectangle (1.5, 1.5);
	\draw[dashed] (-1.5, -1.5) rectangle (1.5, 1.5);
	
	\fill[gray!30,opacity=0.2] (0, -2) -- (0,0) -- (-2, -2) -- cycle;
	\fill[gray!60,opacity=0.4] (0, 0) -- (0, 0.3333) -- (0.3333, 0.3333) -- cycle;
	
	\foreach \x in {-1, ..., 4}
	\foreach \y in {-2, ..., 3}
	\draw (\x,\y) circle (0.08); 
	
	\draw[dotted]
	(0, -2.5) -- (0, 3.5)
	(-2.5, -2.5) -- (4.5, 4.5)
	(-1.5, 0.3333) -- (4.5, 0.3333);
	
	\path (-0.5,-1) node {$H^3_{\mathfrak{m}}$}
	(0,0) node[anchor=west] {$q$};
	\end{scope}
	\end{tikzpicture}
	\caption{The example of Trung and Hoa}
	\label{fig:trunghoa}
\end{figure}

\begin{example}\label{ex:lochom}
	Consider the affine monoid $Q \subset \ZZ^3$ generated by $(0,0,1)$, $(1,0,1)$, $(0,2,1)$, $(1,2,1)$, $(0,3,1)$ and $(1,3,1)$.
	It is shown in the left part of \Cref{fig:trunghoa}.
	This example is taken from \cite{trunghoa}. 
	The holes of $Q$ form a \qq{wall} parallel to the $xz$-plane.
	Hence, nontrivial local cohomology of $\kQ$ can only appear in the degrees of this wall.
	The right part of \Cref{fig:trunghoa} shows this wall and the intersections with the facet defining hyperplanes.
	In each region, $\nab[.]$ and thus $\lochom{i}$ is constant.
	In the shaded unbounded region pointing downwards, we have $\lochom{3} \neq 0$ by \Cref{thm:lochomholes}.
	All other unbounded regions do not support local cohomology because $\lochom{i}$ is Artininan. 
	So the only part of the local cohomology that is not classified so far is the lattice point $q$ in the small shaded triangle.
	In fact, one may compute directly that $\dim_\kk \lochom{2}_q = 1$.
\end{example}

\section{Applications}\label{sec:application}

\subsection{Special configurations of holes} 
In this section, we show that various ring-theoretical properties of $\kQ$ correspond to special configurations of the holes in $Q$.
For positive $Q$, the next proposition appeared as Corollary 5.3 in \cite{schenzel1}.
\begin{proposition}\label{prop:H1}
	Let $Q$ be an affine monoid with $\sdim Q \geq 2$.
	Then $\lochom{1}_q \neq 0$ if and only if $q$ is contained in a zero-*dimensional family of holes. In this case, $\lochom{1}_q = \kk$ and $\lochom{i}_q = 0$ for $i \neq 1$.
\end{proposition}
\begin{proof}
	By \Cref{thm:lochomneu}, we have $\lochom{1} \cong H_{\mm}^{0}(\Quot)$.
	It follows from the proof of \Cref{thm:ZerlegungNeu} that the support of $H_{\mm}^{0}(\Quot)$ is the union of the zero-*dimensional families of holes.
	
	For the second claim, note that $H_{\mm}^{0}(\Quot)$ is a submodule of $\Quot$ and $\dim_\kk (\Quot)_q \leq 1$ for all $q \in \gp{Q}$.
	Moreover, if $\lochom{i}_q \neq 0$ for some $i > 1$, then by \Cref{thm:lochomholes} $q$ has to be contained in a family of holes $\gp{G} + p$ of *dimension $i-1 > 0$.
	But then the zero-*dimensional family of holes containing $q$ is also contained in $\gp{G} + p$ (because the minimal face $F_0$ is contained in $G$), contradiction the irredundancy statement of \ref{thm:ZerlegungNeu}.
\end{proof}

\begin{theorem}\label{thm:compos}
	Let $Q$ be an affine monoid of *dimension $d$. The following holds: 
	\begin{itemize}
		\item If $d \geq 2$, then $\sdep \kQ = 1$       if and only if there is a $0$-*dimensional family of holes.
		\item $Q$ is locally normal                     if and only if there is no family of holes of positive *dimension.
		\item $\kQ$ satisfies Serre's condition $(R_1)$ if and only if there is no family of holes of *dimension $d-1$.
		\item $\kQ$ satisfies Serre's condition $(S_2)$ if and only if every family of holes has *dimension $d-1$.
	\end{itemize}
\end{theorem}
\begin{proof}
	The statements about local normality and Serre's $(R_1)$ follow readily from our description of the families of holes as the associated primes of $\Quot$.
	For proving the criterion for $\sdep \kQ = 1$, recall that if $\sdim Q \geq 2$, then $\sdep \kQ = 1$ if and only if $\lochom{1} \neq 0$. This is equivalent to the existence of a zero-*dimensional family of holes by \Cref{prop:H1}.

	For Serre's criterion $(S_2)$, note that for any face $F$ of $Q$, it holds that $\sdep \kQ[Q_F] = 1$ if and only if $\dep \QuotF = 0$ (cf. \Cref{prop:dep}) and this is in turn equivalent to $\pF$ being an associated prime of $\Quot$.
	On the other hand, it holds that $\sdim \kQ[Q_F] = 1$ if and only if $F$ is a facet of $Q$.
	Therefore, Serre's condition $(S_2)$ is satisfied if and only if all associated primes of $\Quot$ are of dimension $d-1$.
\end{proof}

\begin{remark}
It is well-known that the monoid algebra $\kQ$ satisfies Serre's condition $(S_2)$ if and only if
\begin{equation}\label{eq:S2sch}
Q = \bigcap_{F \text{ facet of } Q} Q_F \,,
\end{equation}
cf. \cite{ishida,schenzel1}.
To see that this is equivalent to our condition, consider following the chain of inclusions:
\begin{equation}
Q \subseteq \bigcap_{F \in \FF{1}} Q_F \subseteq \bigcap_{F \in \FF{2}} Q_F \subseteq
\dotsc \subseteq \bigcap_{F \in \FF{d-1}} Q_F \subseteq \sat{Q} \,.
\end{equation}
It is not difficult to prove that the $i$th inclusion is proper if and only if there is an $F$ of $\sdim F = i$ that corresponds to a family of holes, i.e. $\pF$ is an associated prime of $\Quot$. Hence \eqref{eq:S2sch} holds if and only if all families of holes have *dimension $d-1$.

Moreover, this chain of inclusions gives rise to a similar chain on the algebra $\kQ$ and also on the quotients modulo $\kQ$.
This yields a filtration of $\Quot$ that turns out to be the dimension filtration, cf. \cite{schenzel2}.
It follows that if the $\kQ$-module $\Quot$ is sequentially Cohen-Macaulay, then *depth of $\Quot$ equals the smallest non-zero component in the filtration.
In view of \Cref{prop:dep}, this means that the *depth of $\kQ$ is one more than the smallest *dimension of a family of holes.
\end{remark}

The first part of the \Cref{thm:compos} can be generalized to an upper bound on the *depth:
\begin{theorem} \label{thm:main1}
	If $Q$ has an $i$-*dimensional family of holes, then the *depth of $\kQ$ is at most $i+1$.
\end{theorem}
\begin{proof}
	If $Q$ has an $i$-*dimensional family of holes, then $\lochom{i+1} \neq 0$ by \Cref{thm:lochomholes}.
	Hence $\sdep \kQ \leq i+1$.
	Alternatively, by \Cref{prop:dep} we can consider the depth of $\Quot$.
	As the families of holes of $Q$ correspond to the associated primes of $\Quot$, the claim follows from the general fact that the *depth of a module is bounded above by the *dimensions of its associated primes
	.
\end{proof}

\subsection{Seminormal affine monoids}
In this subsection, we apply our results to seminormal affine monoids.
This way we reprove and extend some results of \cite{brunsgubel}.
Recall that an affine monoid $Q$ is called \emph{seminormal} if $2q, 3q \in Q$ implies $q \in Q$ for $q \in \gp{Q}$.
Equivalently, for every $q \in \QoQ$, the set $\set{m \in \NN \with m q \in Q}$ is contained in a proper subgroup of $\ZZ$.
First, we give a geometric characterization of seminormality that is similar in spirit to the characterizations given in \cite[p. 66f]{brunsgubel}.
\begin{proposition}\label{prop:sngeo}
	Let $Q$ be an affine monoid. $Q$ is seminormal if and only if for every family of holes $\gp{F} + q$ it holds that $q \in \QQ F$.
\end{proposition}
\noindent Here $\QQ F$ denotes the $\QQ$-subspace of $\QQ Q$ generated by $F$.
\begin{proof}
	First, assume that the condition in the statement is satisfied.
	Consider a family of holes $\gp{F} + q$.
	Since $q \in \QoQ$, there exists an $m \in \NN$ such that $m q \in Q$.
	By our assumption, it holds that $m q \in F$ and therefore $j m q + q \in \gp{F} + q \cap \sat{Q} \subset \QoQ$ for every $j \in \NN$.
	It follows that either $2q \notin Q$ or $3q \notin Q$.
	Thus, $Q$ is seminormal.
	
	On the other hand, assume there is a family of holes $\gp{F} + q$ such that $q \not \in \QQ F$. Then there exists an element $p \in \gp{F} + q$ such that $p \in \intt \sat{G}$ and $p \notin G$ for some face $G \supset F$. Thus $Q$ is not seminormal by \cite[Proposition 2.40]{brunsgubel}.
\end{proof}
Note the preceding proposition immediately implies that seminormality is preserved under localization.
We obtain a short proof of Corollary 5.4 of \cite{blrSeminorm}:
\begin{corollary}[\cite{blrSeminorm}] 
	Let $Q$ be a seminormal positive affine monoid of dimension at least $2$. Then $\dep \kQ \geq 2$.
\end{corollary}
\begin{proof}
	If $Q$ is positive, then the minimal face $F_0$ contains only the origin $0 \in \gp{Q}$.
	By \Cref{prop:sngeo} every $0$-dimensional family of holes would be contained in $\QQ F_0 = \set{0} \subset Q$, 
	so there is no $0$-dimensional family of holes.
	Hence the claim follows from \Cref{thm:compos}.
\end{proof}
This result is not valid if one omits the requirement that $Q$ is positive.
\begin{example}\label{ex:seminonpos}
Consider $Q \subset \ZZ^3$ defined by
\[ Q = \set{(x,y,z) \in \ZZ^3 \with x,y \geq 0, z \text{ even or }x>0 \text{ or } y > 0} \]
This monoid is seminormal, has *dimension $2$ and has a $0$-*dimensional family of holes, namely the odd points on the $z$-axis. So it has $\sdep \kQ = 1$ by \Cref{thm:compos}.
\end{example}

Next we give a preliminary characterization of seminormality. Geometrically, we show that the graded components of the local cohomology of a seminormal affine monoid are, in a certain sense, constant on rays from the origin.
\begin{lemma}\label{lemma:ray}
	An affine monoid $Q$ is seminormal if and only if it satisfies the following condition:
	For every $q \in \gp{Q}$ there exists a positive $m \in \NN$ such that for every $j \in \NN$ and every $i \in \NN$ it holds that $\lochom{i}_q \cong \lochom{i}_{(1+m j) q}$ (as $\kk$-vector space).
\end{lemma}
\begin{proof}
	Assume that $Q$ is seminormal and fix an element $q \in \gp{Q}$.
	We will find an $m \in \NN$ such that $\nab = \nab[(1+m j) q]$ for every $j \in \NN$. This implies our claim by \eqref{eq:nabla}.
	First, note that $q \in Q_F$ implies $m q \in Q_F$ for every $m \in \NN$.
	Similarly, $q \notin \sat{Q}_F$ implies $m q \notin \sat{Q}_F$ for every $m \in \NN$.
	So it remains to show the following: There exists an $m \in \NN$, such that for every face $F$ with $q \in \QoQF$ and every $j \in \NN$ it holds that $(1 + j m) q \in \QoQF$.
	For this, consider a face $F$ of $Q$ such that $q \in \QoQF$.
	The localization $Q_F$ is seminormal and thus the set $\set{ m \in \NN \with m q \in Q_F}$ is contained in a proper subgroup of $\ZZ$.
	Since there are only finitely many such faces, we can choose an $m$ in the intersection of these subgroups (for example, the product of the generators).
	Then $1 + j m$ is not contained in any of these subgroups for every $j \in \NN$. Whence our claim follows.
	
	For the converse, assume that $Q$ is not seminormal.
	Let $\gp{F} + q$ be a family of holes such that $q \notin \QQ F$.
	Then there exists a facet $G \supset F$ of $Q$ such that $\sigma_G(q) > 0$.
	By \Cref{thm:lochomholes}, we can find an $p \in \gp{F} + q$ such that $\lochom{i}_p \neq 0$ for $i = \sdim F + 1$.
	By assumption, there exists an $m\in\NN$ such that for all $j \in \NN$ it holds that $\lochom{i}_{(1+m j)p} \neq 0$.
	But this contradicts the fact that $\lochom{i}$ is Artinian, because $\sigma_G(p) > 0$ implies that $\lochom{i}_{(1+m j)p}$ is not contained in the submodule generated by $\lochom{i}_{(1+m j)p + p}$.
\end{proof}
\noindent With a little more work, one can show that the $m$ in the preceding lemma can be chosen independently of $q$.
We extend the characterization of seminormality given in Theorem 4.7 of \cite{blrSeminorm}.
\begin{theorem} 
	\label{thm:charseminormal}
	Let $Q$ be an affine monoid. The following statements are equivalent:
	\begin{enumerate}
		\item $Q$ is seminormal.
		\item $\lochom{i}_q = 0$ for all $q \in \gp{Q}$ such that $q \notin - \sat{Q}$ and all $i$.
		\item $\lochom{i}_q = 0$ for all $q \in \gp{Q}$ such that $q \notin - \sat{Q}$ and all $i$ such that $Q$ has an $(i+1)$-*dimensional family of holes.
	\end{enumerate}
\end{theorem}
\noindent Note that the third condition generalizes Theorem 4.9 in \cite{blrSeminorm}.
\begin{proof}
	1) $\Rightarrow$ 2) 
	Let $Q$ be seminormal and let $q \in \gp{Q}$. By \Cref{lemma:ray}, there exists a positive integer $m$ such that $\lochom{i}_q \cong \lochom{i}_{(1+m j) q}$ for every $i$ and every $j \in \NN$.
	If $q \notin - \sat{Q}$ then $\lochom{i}_q = 0$, because $\lochom{i}$ is Artinian.
	
	2) $\Rightarrow$ 3)
	This is obvious.
	
	3) $\Rightarrow$ 1)
	Assume that $Q$ is not seminormal. Then, by \Cref{prop:sngeo}, there is a family of holes $\gp{F} + q$ of $Q$ such that $q \notin \QQ F$.
	There exists a facet $G$ containing $F$ such that $\sigma_G(q) > 0$.
	By \Cref{thm:lochomholes}, there exists an element $p \in \gp{F} + q$ such that $\lochom{i}_p \neq 0$ where $i = \sdim F + 1$.
	The linear form $\sigma_G$ is constant on $\gp{F} + q$, so $\sigma_G(p) > 0$ and hence $p \notin - \sat{Q}$.
\end{proof}
\newcommand{\FQsat}{F_0^{\sat{Q}}}
Our next results extend Proposition 4.15 of \cite{blrSeminorm}.
In the non-positive case we may have nontrivial local cohomology in the degrees $\FQsat \defa \sat{Q} \cap (-\sat{Q}) = \QQ F_0 \cap \gp{Q}$,
as in \Cref{ex:seminonpos}.
Note that if $Q$ is positive, then $\FQsat = \set{0} \subset Q$, so there can be no local cohomology supported in $\FQsat$.
Moreover, if $\lochom{i}_q \neq 0$ for some $q \in \gp{Q}$ then $q \in - \sat{Q}$ by the preceding theorem. So in this case, $q \notin \FQsat$ if and only if $q \notin \sat{Q}$.
\begin{proposition}[Prop. 4.14, \cite{blrSeminorm}]\label{cor:sninfinite}
	Let $Q$ be a seminormal affine monoid.
	If $\lochom{i}_q \neq 0$ for some $q \in \gp{Q}$, $q \notin \FQsat$, then $\lochom{i}$ is not finitely generated.
\end{proposition}
\begin{proof}
	Assume to the contrary that $\lochom{i}$ is finitely generated, say, in degrees $p_1,\dotsc, p_l \in \gp{Q}$.
	We assumed that $q \notin \sat{Q}$, so there exists a facet $F$ such that $\sigma_F(q) < 0$.
	By \Cref{lemma:ray}, there is an $m \in \NN$ such that $\lochom{i}_{(1+mj)q} \neq 0$ for every $j \in \NN$.
	For sufficiently large $j \in \NN$, it holds that $\sigma_F((1+mj)q) < \sigma_F(p_k)$ for every $k$, so the graded component in this degree cannot be generated by our supposed set of generators, a contradiction.
\end{proof}
\begin{corollary}\label{cor:CMlocal}
	Let $Q$ be a seminormal affine monoid such that $\lochom{i}_q = 0$ for every $q \in \FQsat$ and every $i < \sdim Q$. Then
	\[ \sdep \kQ = \min\set{\sdep \kQ[Q_F] + 1 \with F \text{ a face, } \sdim F = 1}\,. \]
	In particular, $\kQ$ is Cohen-Macaulay if and only if it is locally Cohen-Macaulay.
\end{corollary}
\begin{proof}
	Our hypothesis implies that all the non-vanishing local cohomology modules of $\kQ$ are not finitely generated.
	Therefore, the claim is a consequence of the graded Finiteness Theorem, cf. \cite[Theorem 14.3.10]{brodmannsharp2} and Exercise 9.5.4 (i) in the same book.
\end{proof}

\subsection{Cases of equality}
In this subsection, we identify several cases where the depth of $\kQ$ equals the upper bound given by \Cref{thm:main1}.
At first, we show that this is the case if the holes are of the simplest possible kind.
\begin{proposition}\label{thm:main2}
	If $\QoQ = q + F$ for an element $q \in \QoQ$ and a face $F$ of $Q$, then $\sdep \kQ = 1 + \sdim F$.
\end{proposition}
Note that it is not sufficient to require that there is only one family of holes, see \Cref{ex:lochom}.

\begin{proof}[Proof of \Cref{thm:main2}]
	The hypothesis is equivalent to the statement that $\Quot \cong {\kQ}/{\pF} = \kQ[F]$ (up to a shift in the grading).
	Recall that the *depth of a module $M$ over a ring $R$ equals its *depth over $R / \Ann M$.
	Together with \Cref{prop:dep} this yields
	\[ \sdep_{\kQ} \kQ = 1 + \sdep_{\kQ[F]} \kQ[F] \,.\]
	
	Next, note that $\sat{F} + q \subseteq \QoQ$. To the contrary, assume that there is an element $f \in \sat{F}$ such that $f + q \in Q$. We can write $f = f_1 - f_2$ with $f_1, f_2 \in F$. But $f + f_2 + q = f_1 + q\in \QoQ$ by assumption, a contradiction.
	
	Therefore we conclude that $\QoQ = F + q \subseteq \sat{F} + q \subseteq \QoQ$, so $F = \sat{F}$.
	Hence $\kQ[F]$ is normal and thus Cohen-Macaulay.
\end{proof}
\noindent We give an example how one can effectively compute the *depth using \Cref{thm:main2}.
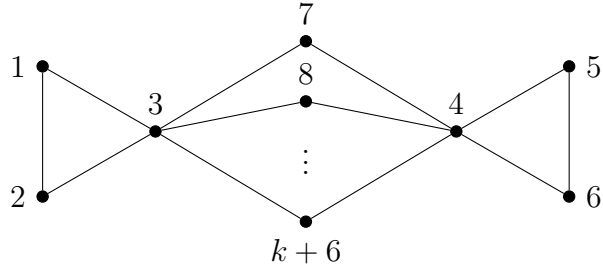
\begin{figure}[t]
	\centering
%
	\begin{tikzpicture}[scale=1, normal/.style={circle,inner sep=0pt, fill=black,  minimum size=1.5mm, draw}]
	\path (-3,0)
	node[normal, label=left:$1$]  (1) at +(120:1) {}
	node[normal, label=left:$2$]  (2) at +(240:1) {}
	node[normal, label=above:$3$] (3) at +(  0:1) {};
	
	\path (3,0)
	node[normal, label=above:$4$] (4) at +(180:1) {}
	node[normal, label=right:$5$] (5) at +( 60:1) {}
	node[normal, label=right:$6$] (6) at +(300:1) {};
	
	\path (0,0)
	node[normal, label=above:$7$]   (7) at +(0, 1.2) {}
	node[normal, label=above:$8$]   (8) at +(0, 0.4) {}
	node                                at +(0,-0.3) {\vdots}
	node[normal, label=below:$k+6$] (9) at +(0,-1.2) {};
	
	\draw (1) -- (2) -- (3) -- (1);
	\draw (4) -- (5) -- (6) -- (4);
	\draw (3) -- (7) -- (4);
	\draw (3) -- (8) -- (4);
	\draw (3) -- (9) -- (4);
	\end{tikzpicture}
	\caption{The graph $G_{k+6}$}
	\label{fig:Gkgraph}
\end{figure}
\begin{example}\label{ex:graph}
	\newcommand{\ee}{\mathbf e}
	Let $G$ be a graph with vertex set $V$ and edge set $E$. We associate an affine monoid to $G$, the \emph{toric edge ring} $k[G]$, introduced in \cite{OH98}. This is the monoid algebra to the monoid generated by the vectors $\ee_i + \ee_j \in \ZZ^{\#V}$, where $\set{i,j}$ is an edge of $G$ and $\ee_i, \ee_j$ denotes unit vectors indexed by the vertices of $G$.
	For positive $k \in \NN$ consider the graph $G_{k+6}$ in \Cref{fig:Gkgraph}.
	In \cite{HHKO} the *depth of the toric edge ring of this family of graphs is computed.
	We will show that these edge rings satisfy the assumption of \Cref{thm:main2},
	and thus give an alternative computation of the *depth.
	
	First, it is known that $\kQsat$ is generated as a $\kQ$-module by $x_1 x_2 x_3 x_4 x_5 x_6$, i.e. the monomial corresponding to the vector $q \in \sat{Q} \subseteq \RR^{k+6}$ which assigns $1$ to the vertices $1,\dotsc,6$.
	If we add one of the \qq{middle} edges, e.g. $\set{3,8}$, to $q$, then it is easy to see that the result lies in $Q$.
	On the other hand, if we add any combination of edges from the triangles to $q$, then the result will always be in $\QoQ$. To see this, note that the sum over the vertices of each triangle is always odd.
	This implies that $\QoQ = F + q$, where $F$ is the face spanned by the six edges in the triangles.
	The dimension of $F$ is $6$, so by \Cref{thm:main2} it follows that
	$\dep \kQ = 1 + 6 = 7$.
\end{example}

If we relax the restriction on the holes, then we need additional assumptions on the affine monoid $Q$.
The next two results are of this kind.
Recall that an affine monoid $Q$ is called \emph{simplicial} if its cross section polytope is a simplex.
Equivalently, $Q$ is simplicial if its face lattice is a boolean lattice.
Sometimes simplicial affine monoids are required to be positive, but we allow non-positive simplicial affine monoids.
Moreover, we call $Q$ \emph{locally simplicial}, if all its localisations are simplicial. Equivalently, $Q$ is locally simplicial if its cross section polytope is a simple polytope.

\begin{theorem}\label{thm:disjoint}
 	Assume that $Q$ is either simplicial or seminormal.
 	If the families of holes of $Q$ are pairwise disjoint, then the *depth of $\kQ$ equals one plus the smallest *dimension of a family of holes.
\end{theorem}

\begin{theorem}\label{thm:S2}
	Assume that $Q$ satisfies $(S_2)$. If $Q$ is
	\begin{itemize}
		\item either simplicial
		\item or seminormal, locally simplicial and $\lochom{i}_q = 0$ for every $q \in \FQsat$ and every $i < \sdim Q$,
	\end{itemize}
	then $\kQ$ is Cohen-Macaulay.
\end{theorem}
The first item in \Cref{thm:S2} generalized the well-known result that positive, simplicial affine monoids are Cohen-Macaualy, if they satisfy $(S_2)$ (cf. \cite{gws}) while the second item generalizes \cite[Cor. 5.6]{blrSeminorm}. For this, note the condition on the local cohomology is always satisfies it $Q$ is positive.

\begin{remark}
	Despite the technical assumptions in the preceding three results, they seem to be best possible in the following sense.
	There are affine monoids whose depth is strictly smaller than the bound of \Cref{thm:main1}, which satisfy the any of the following combinations of properties:
	\begin{enumerate}
		\item positive, locally simplicial, $(S_2)$, only one family of holes (\Cref{ex:lochom}, cf. \cite{trunghoa});
		\item positive, seminormal, simpicial; 
		\item positive, seminormal, $(S_2)$ (cf. \cite{blrSeminorm});
		\item seminormal, locally simplicial, $(S_2)$
	\end{enumerate}
	In all cases except the first one, the depth may even depend on the characteristic of the field.
	Note that in the last case $Q$ will be at least locally Cohen-Macaulay by \Cref{thm:S2}.
\end{remark}

We need some preparations for the proofs of \Cref{thm:disjoint} and \Cref{thm:S2}.
Recall that we defined $\nab$ to be the set of faces of $Q$ such that $q \in Q_F$ for $q \in \gp{Q}$.
Let $\nabv$ be the corresponding subset of faces of the polytope $\Pc^\vee$ polar to the cross-section polytope $\Pc$ of $\cone{Q}$.
As $\nab$ is closed under taking supersets, $\nabv$ is closed under taking subsets and hence a polytopal subcomplex of $\Pc^\vee$.
Let $\snab$ denote the set of faces of $Q$ such that $q \in \sat{Q}_F$.
\begin{lemma}\label{lemma:snab}
	Let $q \in \gp{Q}$.
	Assume that either $Q$ is simplical or $q \in -\sat{Q}$.
	Then $\snab$ has a unique minimal element.
	Equivalently, $\snabv$ is (the boundary complex of) a face of $\Pc^\vee$.
\end{lemma}
\begin{proof}
	The claim is a statement about $\sat{Q}$, so we may assume that $Q = \sat{Q}$.
	
	First we consider the case that $Q$ is simplicial.
 	We write $\Fc_{\geq}$ for the set of facets $F$ of $Q$ such that $\sigma_F(q) \geq 0$ and we write $\Fc_{<}$ for the set of facets $F$ such that $\sigma_F(q) < 0$.
 	Our candidate for the unique minimal element of $\nab$ is the intersection $G$ of the facets in $\Fc_{\geq}$.
 	This is indeed a face because $Q$ is simplicial. 
 	If $p \in \intt G$ is an interior element, then by construction $\sigma_F(p) > 0$ for all $F$ in $\Fc_{<}$.
 	Therefore, $q + m p \in Q$ for $m \gg 0$ and hence $q \in Q_G$ and thus $G \in \nab$.
 	On the other hand, let $G'\in \nab$ be face. Then $q \in Q_{G'}$,
 	so there exists an element $g \in G'$ such that $q+g \in Q$.
 	It follows that $\sigma_F(g) > 0$ for all $F \in \Fc_{<}$.
 	Hence $G'$ is not contained in any facet in $\Fc_{<}$ and can therefore be written as an intersection of facets in $\Fc_{\geq}$.
 	It follows that $G \subset G'$, so $G$ is minimal.
 
	Next, we consider the case that $q \in -Q$.
 	There exists a unique face $G$ such that $q \in - \intt G$.
 	Evidently $q \in Q_G$.
 	We show that this $G$ is the unique minimal element of $\nab$.
 	So let $F$ be a facet of $Q$ that does not contain $G$.
 	Then $\sigma_F(q) < 0$ and hence $q \notin Q_F$.
 	The same holds then for every face contained in $F$.
 	It follows that every face $G' \in \nab$ is contained only in the facets containing $G$.
 	But $G$ is the intersection of the facets containing it, hence $G \subseteq G'$.
 	So $G$ is the unique minimal element of $\nab$.
\end{proof}
We recall the formula \eqref{eq:nabla} from above
\begin{equation*}
\lochom{i}_q = \tilde{H}_{d-1-i}(\nabv, \kk)
\end{equation*}
as it will be the key ingredient in the sequel.

\begin{proof}[Proof of \Cref{thm:disjoint}]
	Consider an element $q \in \gp{Q}$.
	If $Q$ is seminormal and $q \notin -\sat{Q}$, then $\lochom{i}_q = 0$ for all $i$ by \Cref{thm:charseminormal}.
	Otherwise we are in the situation of the preceding lemma, so $\snabv$ is the boundary complex of a polytope.
	By assumption, $q$ is contained in at most one family of holes, so $\nabv$ is obtained from $\snabv$ be removing at most one face (and all faces containing it).
	It follows that $\nabv$ is either a ball or a sphere, so the claim follows from \eqref{eq:nabla}.
\end{proof}

\begin{proof}[Proof of \Cref{thm:S2}]
	Let $q \in \gp{Q}$. Again, we may restrict our attention ot the situation of \Cref{lemma:snab}.
	
	We show that $\snabv$ is a (full) simplex.
	If $Q$ is simplicial, then $\Pc$ and (thus) $\Pc^\vee$ are simplices, so $\snabv$ is also a simplex.
	On the other hand, if $Q$ is locally simplicial, then $\Pc$ is a simple polytope and hence $\Pc^\vee$ is a simplicial polytope.
	Moreover, we assumed that $\lochom{i}$ vanishes on $\FQsat$, so we may assume that $q \in -\intt F$ for a face $F \subset Q$ with positive *dimension.
	Hence $\snabv$ is a proper face of $\Pc^\vee$ and thus a simplex.
	
	To conclude the proof, recall that $(S_2)$ is equivalent to the statement that all families of holes have *dimension $d-1$.
	So $\nabv$ is obtained from $\snabv$ by removing some vertices.
	Hence $\snabv$ is either a simplex and thus acyclic, or empty with homology in degree $-1$. Now the claim follows from \eqref{eq:nabla}.
\end{proof}

There is an alternative way to derive the seminormal case of \Cref{thm:S2} from the simplicial case.
Assume that $Q$ is locally simplicial, seminormal, satisfies $(S_2)$ and $\lochom{i}_q = 0$ for every $q \in \FQsat$.
Then $\kQ$ is locally Cohen-Macaulay by \Cref{thm:S2} and hence Cohen-Macaulay by \Cref{cor:CMlocal}.

\printbibliography{}
\end{document}